\documentclass[11pt]{amsart}
\makeatletter
\def\subsection{\@startsection{subsection}{3}
  \z@{.5\linespacing\@plus.7\linespacing}{.1\linespacing}
  {\normalfont\itshape}}
\makeatother

\usepackage{amssymb}
\usepackage{amsmath}
\usepackage{mathrsfs}
\usepackage{amsthm}
\usepackage{enumerate}
\newtheorem{theorem}{Theorem}[section]
\newtheorem{lemma}[theorem]{Lemma}

\theoremstyle{definition}

\numberwithin{equation}{section}
\mathchardef\hyphen="2D
\def\@tvsp{\mathchoice{{}\mkern-4.5mu}{{}\mkern-4.5mu}{{}\mkern-2.5mu}{}}
\def\ln{\left|\@tvsp\left|\@tvsp\left|}
\def\rn{\right|\@tvsp\right|\@tvsp\right|}
\makeatother
\begin{document}
\title{Maximum momemts of sum of independent random matrices}
\author{March T.~Boedihardjo}
\address{Department of Mathematics, Texas A\&M University, College Station, Texas 77843}
\email{march@math.tamu.edu}
\subjclass[2010]{}
\begin{abstract}
We show that the maximum moments of the sum of independent positive semidefinite random matrices with given norm upper bounds and norms of expectations is attained
when all the random matrices are the multiplications of certain random variables and the identity matrix.
\end{abstract}
\maketitle
\section{Main result}
\begin{theorem}\label{11}
Let $p,n,N\geq 1$ be natural numbers. Let $L_{1},\ldots,L_{N}>0$ and $\alpha_{1},\ldots,\alpha_{N}\in[0,1]$. Among all independent, positive semidefinite $n\times n$ random matrices $X_{1},\ldots,X_{N}$ satisfying $\|X_{k}\|\leq L_{k}$ and $\displaystyle\|\mathbb{E}X_{k}\|=\alpha_{k}L_{k}$ for all $k=1,\ldots,N$, the quantity
\[\mathbb{E}\circ\mathrm{tr}\left(\sum_{k=1}^{N}X_{k}\right)^{p}\]
is maximized when $\mathbb{P}(X_{k}=L_{k}I)=\alpha_{k}$ and $\mathbb{P}(X_{k}=0)=1-\alpha_{k}$.
\end{theorem}
In the case when $N=n$, $L_{k}=1$ and $\alpha_{k}=\frac{1}{n}$, by \cite[Theorem 2.5]{Johnson}, it follows that for all independent, positive semidefinite $n\times n$
random matrices $X_{1},\ldots,X_{n}$ satisfying $\|X_{k}\|\leq 1$ and $\|\mathbb{E}X_{k}\|=\frac{1}{n}$ for all $k=1,\ldots,n$,
\[\mathbb{E}\circ\mathrm{tr}\left(\sum_{k=1}^{n}X_{k}\right)^{p}\leq\left(C\frac{p}{\log p}\right)^{p}n,\]
where $C$ is a universal constant.
\section{Proof of the main result}
\begin{lemma}[H\"older inequality, \cite{Bhatia}, Corollary IV.2.6]\label{21}
Let $p_{1},\ldots,p_{r}\geq 1$ with $\frac{1}{p_{1}}+\ldots+\frac{1}{p_{r}}=1$. Let $A_{1},\ldots,A_{r}$ be $n\times n$ matrices. Then
\[|A_{1}\ldots A_{r}|_{1}\leq|A_{1}|_{p_{1}}\ldots|A_{r}|_{p_{r}}\]
\end{lemma}
\begin{lemma}[Araki-Lieb-Thirring inequality, \cite{Bhatia}, Exercise IX.2.11]
Let $\alpha\geq 1$. Let $A,B$ be positive semidefinite $n\times n$ matrices. Then
\[\mathrm{tr}(ABA)^{\alpha}\leq\mathrm{tr}(A^{2\alpha}B^{\alpha}).\]
\end{lemma}
As an immediate consequence, we have
\begin{lemma}\label{23}
Let $\alpha\geq 1$. Let $A,B$ be positive semidefinite $n\times n$ matrices. Then
\[|ABA|_{\alpha}\leq\left(\mathrm{tr}(A^{2\alpha}B^{\alpha})\right)^{\frac{1}{\alpha}}.\]
\end{lemma}
\begin{lemma}\label{24}
Let $l_{1},\ldots,l_{r},m_{1},\ldots,m_{r}\geq 1$. Let $l=l_{1}+\ldots+l_{r}$ and $m=m_{1}+\ldots+m_{r}$. Let $X,Y$ be positive semidefinite $n\times n$ matrices. Then
\[|\mathrm{tr}X^{l_{1}}Y^{m_{1}}\ldots X^{l_{r}}Y^{m_{r}}|\leq\|X\|^{l-1}\mathrm{tr}XY^{m}.\]
\end{lemma}
\begin{proof}
\begin{eqnarray*}
|\mathrm{tr}X^{l_{1}}Y^{m_{1}}X^{l_{2}}Y^{m_{2}}\ldots X^{l_{r}}Y^{m_{r}}|&=&
|\mathrm{tr}X^{\frac{l_{1}}{2}}X^{\frac{l_{1}}{2}}Y^{m_{1}}X^{\frac{l_{2}}{2}}X^{\frac{l_{2}}{2}}Y^{m_{2}}\ldots X^{\frac{l_{r}}{2}}X^{\frac{l_{r}}{2}}Y^{l_{r}}|
\\&=&|\mathrm{tr}X^{\frac{l_{1}}{2}}Y^{m_{1}}X^{\frac{l_{2}}{2}}X^{\frac{l_{2}}{2}}Y^{m_{2}}\ldots X^{\frac{l_{r}}{2}}X^{\frac{l_{r}}{2}}Y^{l_{r}}X^{\frac{l_{1}}{
2}}|\\&\leq&|X^{\frac{l_{1}}{2}}Y^{m_{1}}X^{\frac{l_{2}}{2}}X^{\frac{l_{2}}{2}}Y^{m_{2}}\ldots X^{\frac{l_{r}}{2}}X^{\frac{l_{r}}{2}}Y^{l_{r}}X^{\frac{l_{1}}{2}}|_{1}\\&\leq&|X^{\frac{l_{1}}{2}}Y^{m_{1}}X^{\frac{l_{2}}{2}}|_{\frac{m}{m_{1}}}|X^{\frac{l_{2}}{2}}Y^{m_{2}}X^{\frac{l_{3}}{2}}|_{\frac{m}{m_{2}}}\ldots|X^{\frac{l_{r}}{2}}Y^{m_{r}}X^{\frac{l_{1}}{2}}|_{\frac{m}{m_{r}}}.
\end{eqnarray*}
To get the first equality, we write $X^{l_{1}}=X^{\frac{l_{1}}{2}}X^{\frac{l_{1}}{2}}$, $X^{l_{2}}=X^{\frac{l_{2}}{2}}X^{\frac{l_{2}}{2}}$ etc. The second equality follows by moving $X^{\frac{l_{1}}{2}}$ to the end of the the product in the trace. The last inequality follows from Lemma \ref{21}.

To simplify the notation, let $l_{r+1}=l_{1}$. Then we obtain
\[|\mathrm{tr}X^{m_{1}}Y^{m_{1}}\ldots X^{l_{r}}Y^{m_{r}}|\leq\prod_{i=1}^{r}|X^{\frac{l_{i}}{2}}Y^{m_{i}}X^{\frac{l_{i}}{2}}|_{\frac{m}{m_{i}}}.\]
For each $1\leq i\leq r$, we have
\[|X^{\frac{l_{i}}{2}}Y^{m_{i}}X^{\frac{l_{i+1}}{2}}|_{\frac{m}{m_{i}}}\leq\|X\|^{\frac{l_{i}}{2}-\frac{m_{i}}{2m}+\frac{l_{i+1}}{2}-\frac{m_{i}}{2m}}|X^{\frac{m_{i}}{2m}}Y^{m_{i}}X^{\frac{m_{i}}{2m}}|_{\frac{m}{m_{i}}}\leq\|X\|^{\frac{l_{i}+l_{i+1}}{2}-\frac{m_{i}}{m}}\left(\mathrm{tr}XY^{m}\right)^{\frac{m_{i}}{m}},\]
where the first inequality follows from the ideal property of the Schatten norm and the second inequality follows from Lemma \ref{23}. Therefore,
\[\mathrm{tr}X^{l_{1}}Y^{m_{1}}\ldots X^{l_{r}}Y^{m_{r}}\leq\prod_{i=1}^{r}\left(\|X\|^{\frac{l_{i}+l_{i+1}}{2}-\frac{m_{i}}{m}}\left(\mathrm{tr}XY^{m}\right)^{\frac{m_{i}}{m}}\right).\]
The sum of the powers of $\|X\|$ is given by
\[\sum_{i=1}^{r}\left(\frac{l_{i}+l_{i+1}}{2}-\frac{m_{i}}{m}\right)=\frac{1}{2}\left(\sum_{i=1}^{r}l_{i}+\sum_{i=1}^{r}l_{i+1}\right)-1=\frac{1}{2}(l+l)-1=l-1.\]
We conclude that
\[\mathrm{tr}X^{l_{1}}Y^{m_{1}}\ldots X^{l_{r}}Y^{m_{r}}\leq\|X\|^{l-1}\mathrm{tr}XY^{m}.\]
\end{proof}
\begin{lemma}\label{25}
Let $l_{1},\ldots,l_{r},m_{1},\ldots,m_{r}\geq 1$. Let $l=l_{1}+\ldots+l_{r}$ and $m=m_{1}+\ldots+m_{r}$. Let $X,Y$ be independent, positive semidefinite $n\times n$ random matrices such that $\|X\|\leq L$. Then
\[\mathbb{E}\circ\mathrm{tr}X^{l_{1}}Y^{m_{1}}\ldots X^{l_{r}}Y^{m_{r}}\leq\mathbb{E}\circ\mathrm{tr}f^{l}Y^{m},\]
where $f$ is a random variable on $\{0,L\}$ independent from $Y$ with $\displaystyle\mathbb{P}(f=L)=\frac{\|\mathbb{E}X\|}{L}$.
\end{lemma}
\begin{proof}
By Lemma \ref{24},
\begin{eqnarray*}
\mathbb{E}\circ\mathrm{tr}X^{l_{1}}Y^{m_{1}}\ldots X^{l_{r}}Y^{m_{r}}\leq\mathbb{E}(\|X\|^{l-1}\mathrm{tr}XY^{m})&\leq&L^{l-1}(\mathbb{E}\circ\mathrm{tr}XY^{m})\\&=&L^{l-1}\mathrm{tr}(\mathbb{E}X)(\mathbb{E}Y^{m})\\&\leq&L^{l-1}\|\mathbb{E}X\|\mathrm{tr}(\mathbb{E}Y^{m})\\&=&\mathbb{E}f^{l}\mathbb{E}\circ\mathrm{tr}Y^{m}\\&=&\mathbb{E}\circ\mathrm{tr}f^{l}Y^{m}.
\end{eqnarray*}
\end{proof}
\begin{lemma}\label{26}
Let $p$ be a natural number. Let $X,Y$ be independent, positive semidefinite $n\times n$ random matrices such that $\|X\|\leq L$. Then
\begin{equation}\label{21e}
\mathbb{E}\circ\mathrm{tr}(X+Y)^{p}\leq\mathbb{E}\circ\mathrm{tr}(fI+Y)^{p},
\end{equation}
where $f$ is a random variable on $\{0,L\}$ independent from $Y$ with $\displaystyle\mathbb{P}(f=L)=\frac{\|\mathbb{E}X\|}{L}$. Equality holds when $X$ and $Y$ commute, $X/L$ is a projection, and $\mathbb{E}X$ is a scalar multiple of the identity.
\end{lemma}
\begin{proof}
Since $(X+Y)^{p}$ is a sum of products of the form $X^{l_{1}}Y^{m_{1}}\ldots X^{l_{r}}Y^{m_{r}}$, (\ref{21e}) follows from Lemma \ref{25}. When $X$ and $Y$ commute,
\[\mathbb{E}\circ\mathrm{tr}(X+Y)^{p}=\sum_{s=0}^{p}\binom{p}{s}\mathbb{E}\circ\mathrm{tr}X^{s}Y^{p-s}=\sum_{s=0}^{p}\binom{p}{s}\mathrm{tr}(\mathbb{E}X^{s})(\mathbb{E}Y^{p-s}).\]
When $X/L$ is a projection and $\mathbb{E}X$ is a scalar multiple of the identity, $\mathbb{E}X^{s}=(\mathbb{E}f^{s})I$. Therefore,
\[\mathbb{E}\circ\mathrm{tr}(X+Y)^{p}=\sum_{s=0}^{p}\binom{p}{s}\mathrm{tr}(\mathbb{E}f^{s})(\mathbb{E}Y^{p-s})=\mathbb{E}\circ\mathrm{tr}(fI+Y)^{p}.\]
\end{proof}
\begin{proof}[Proof of Theorem \ref{11}]
Let $f_{1},\ldots,f_{N}$ be independent random variables on $\{0,L_{1}\},\ldots,\{0,L_{N}\}$, respectively, such that $\mathbb{P}(f_{k}=L_{k})=\alpha_{k}$. We may assume that $f_{1},\ldots,f_{N}$ are independent from $X_{1},\ldots,X_{N}$. Applying Lemma \ref{26} recursively, we obtain
\begin{eqnarray*}
\mathbb{E}\circ\mathrm{tr}(X_{1}+\ldots+X_{N})^{p}&=&\mathbb{E}\circ\mathrm{tr}(X_{1}+(X_{2}+\ldots+X_{N}))^{p}\\&\leq&\mathbb{E}\circ(f_{1}I+(X_{2}+\ldots+X_{N}))^{p}\\&=&\mathbb{E}\circ\mathrm{tr}(X_{2}+(X_{3}+\ldots+X_{N}+f_{1}I))^{p}\\&\leq&\mathbb{E}\circ\mathrm{tr}(f_{2}I+(X_{3}+\ldots+X_{N}+f_{1}I))^{p}\\&=&\mathbb{E}\circ\mathrm{tr}(X_{3}+(X_{4}+\ldots+X_{N}+f_{1}I+f_{2}I))^{p}\\&\leq&\ldots\\&\leq&\mathbb{E}\circ\mathrm{tr}(f_{1}I+\ldots+f_{N}I)^{p}\\&=&\mathbb{E}(f_{1}+\ldots+f_{N})^{p},
\end{eqnarray*}
where the first inequality follows by taking $X=X_{1}$ and $Y=X_{2}+\ldots+X_{N}$ in Lemma \ref{26} and the second inequality follows by taking $X=X_{2}$ and $Y=X_{3}+\ldots+X_{N}+f_{1}I$.

Moreover, when $X_{1},\ldots,X_{N}$ commute, $X_{k}/L_{k}$ is a projection and $\mathbb{E}X_{k}$ is a scalar multiple of the identity, all inequalities become equalities. This happens, for
instance, when $\mathbb{P}(X_{k}=L_{k}I)=\alpha_{k}$ and $\mathbb{P}(X_{k}=0)=1-\alpha_{k}$.
\end{proof}

\end{document}